\documentclass[12pt]{amsart}  
\usepackage{amsmath,amsthm,amssymb,amsfonts,eucal, latexsym,mathrsfs, stmaryrd}
\usepackage{comment}
\usepackage{color}
\usepackage{graphicx}
\usepackage{enumerate}
\usepackage{parskip}
\usepackage{mathdots}

\numberwithin{equation}{section}
\theoremstyle{plain}
\newtheorem{Proposition}[equation]{Proposition}
\newtheorem{Corollary}[equation]{Corollary}
\newtheorem*{Corollary*}{Corollary}
\newtheorem{Theorem}[equation]{Theorem}
\newtheorem*{Theorem*}{Theorem}

\theoremstyle{definition}
\newtheorem{Definition}[equation]{Definition}

\newtheorem{Example}[equation]{Example}

\newtheorem{Question}[equation]{Question}

\allowdisplaybreaks

\def\C{\mathbb{C}}

\def\D{\mathbb{D}}
\def\T{\mathbb{T}}

\def\Z{\mathbb{Z}}
\def\phi{\varphi}

\def\K{\mathcal{K}}

\renewcommand{\leq}{\leqslant}
\renewcommand{\geq}{\geqslant}
\renewcommand{\subset}{\subseteq}

\begin{document}
%
%
%%%%%%%%%%%%%%%%%%%%%%%%%%%%%%%%%%%%%%%%%%%%%%%%%%%%%%%%%%%%%%%%%%%%%
%
%
%
%
%
%
%
%
%%%%%%%%%%%%%%%%%%%%%%%%%%%%%%%%%%%%%%%%%%%%%%%%%%%%%%%%%%%%%%%%%%%%%

\title{Inner vectors for Toeplitz operators}

\author[Cheng]{Raymond Cheng}
\address{Department of Mathematics and Statistics,
  Old Dominion University,
  Norfolk, VA 23529}
  \email{rcheng@odu.edu}

\author[Mashreghi]{Javad Mashreghi}
\address{D\'epartement de math\'ematiques et de statistique, Universit\'e laval, Qu\'ebec, QC, Canada, G1V 0A6}
\email{javad.mashreghi@mat.ulaval.ca}

\author[Ross]{William T. Ross}
	\address{Department of Mathematics and Computer Science, University of Richmond, Richmond, VA 23173, USA}
	\email{wross@richmond.edu}
	
	\dedicatory{Dedicated to Thomas Ransford on the occasion of his sixtieth birthday.}

\begin{abstract}
In this paper we survey and bring together several approaches to obtaining inner functions for Toeplitz operators.  These approaches include the classical definition, the Wold decomposition, the operator-valued Poisson Integral, and Clark measures. We then extend these notions somewhat to inner functions on model spaces.  Along the way we present some novel examples. \end{abstract}

%\subjclass[2010]{Primary: , Secondary: }

%\keywords{}

\thanks{This work was supported by NSERC (Canada).}

\maketitle

%%%%%%%%%%%%%%%%%%%%%%%%%%%%%%%%%%%%%%%%%%%%%%%%%%%%%%%%%%%%%%%%%%%%%
%%%%%%%%%%%%%%%%%%%%%%%%%%%%%%%%%%%%%%%%%%%%%%%%%%%%%%%%%%%%%%%%%%%%%
%%%%%%%%%%%%%%%%%%%%%%%%%%%%%%%%%%%%%%%%%%%%%%%%%%%%%%%%%%%%%%%%%%%%%

\bibliographystyle{plain}

\section{Introduction} 
For $\phi \in H^{\infty}$, the bounded analytic functions on the open unit disk $\D$, let 
\begin{equation}\label{98743oie}
T_{\phi}: H^2 \to H^2, \quad T_{\phi} f = \phi f,
\end{equation} denote the analytic {\em Toeplitz operator}  on the classical Hardy space $H^2$.  In this paper we survey, continue, and synthesize some discussions begun in  \cite{InnerRKS, MR1437202, MR1969797} dealing with the notion of an ``inner vector'' for $T_{\phi}$ along with the general notion of an inner vector for a contraction on a Hilbert space. We connect these results with the operator-valued Poisson kernel and some work from \cite{MR2018850, MR2050137} concerning ``factoring an $L^1$ function through a contraction''. Along the way we also produce some interesting examples and reformulations of these connections.

\section{Basic definitions and facts}

We begin with the definition of an inner vector for a Toeplitz operator from \cite{MR1437202}. Recall that the inner product on the Hardy space $H^2$ is 
\begin{equation}\label{H2IP}
\langle f, g \rangle := \int_{\T} f\overline{g} \,dm,
\end{equation}
where $m$ is normalized Lebesgue measure on the unit circle $\T$.  As is tradition, we equate an $f \in H^2$ with its $L^2 = L^2(\T, m)$ radial boundary function, i.e., 
$$f(\zeta) = \lim_{r \to 1^{-}} f(r \zeta)$$
for almost every $\zeta \in \T$.  We will also use the term {\em inner} function (without any qualifiers like in Definition \ref{oow0099898} below) to denote an $H^{\infty}$ function that has unimodular boundary values almost everywhere. Classical theory \cite{Duren} says that an inner function $I$ can be factored uniquely as $I = \xi B S_{\mu}$, where $\xi$ is a unimodular constant, $B$ is a Blaschke product, and $S_{\mu}$ is a singular inner function associated with a positive measure $\mu$ on $\T$ that is singular with respect to $m$. We say the {\em degree} of $I$ is equal to $d$ if $I$ is a finite Blaschke product of order $d$, and equal to infinity otherwise. Furthermore, any function $f \in H^2$ can be factored, uniquely up to multiplicative constants, as $f = I G$, where $I$ is an inner function and $G \in H^2$ is an outer function.

For $\phi \in H^{\infty}$ the analytic Toeplitz operator $T_{\phi}$ from \eqref{98743oie} is a bounded operator on $H^2$ whose norm $\|T_{\phi}\|$ satisfies 
$$\|T_{\phi}\| = \|\phi\|_{\infty} := \operatorname{ess-sup\{}|\phi(\xi)|: \xi \in \T\}.$$
Also recall that the adjoint $T_{\phi}^{*}$ of $T_{\phi}$ satisfies $T_{\phi}^{*} = T_{\overline{\phi}}$, where $T_{\overline{\phi}} f = P(\overline{\phi} f)$ and $P$ is the Riesz projection of $L^2$ onto $H^2$. When $\phi$ is an inner function, observe from \eqref{H2IP} that $T_{\phi}$ is an isometry. See \cite[Ch.~4]{MR3526203} for the details of these basic Toeplitz operator facts and \cite{BS} for a definitive treatise.

\begin{Definition}\label{oow0099898}
For $\phi \in H^{\infty}$ we say a unit vector $f \in H^2$ is {\em $T_{\phi}$-inner} if 
$\langle T^{n}_{\phi} f, f\rangle = 0$ for all $n \geq 1$. 
\end{Definition}

When $\phi(z) = z$, one can see from Fourier analysis that the $T_{z}$-inner vectors are precisely the inner functions. Also observe that replacing $\phi$ with $c \phi$, where $c > 0$, in Definition \ref{oow0099898} does not change whether or not a function $f$ is $T_{\phi}$-inner. Thus we can always assume, by scaling $\phi$, that 
$$\phi \in b(H^{\infty}) := \{g \in H^{\infty}: \|g\|_{\infty} \leq 1\},$$
the closed unit ball of $H^{\infty}$. This normalization will be important when we need $T_{\phi}$ to be a contraction operator since in this case $\|T_{\phi}\| = \|\phi\|_{\infty} \leq 1$. 
 Immediate from Definition \ref{oow0099898} and the inner product formula from \eqref{H2IP} are the following facts. 

\begin{Proposition}\label{6iehtg43tnnncknkn}
Let $\phi \in b(H^{\infty})$. 
\begin{enumerate}
\item If $f \in H^2$ is $T_{\phi}$-inner and $I$ is any inner function, then $I f$ is $T_{\phi}$-inner. 
\item If $f \in H^2$ is $T_{\phi}$-inner and $\Theta$ is any inner divisor of $f$, i.e., $f/\Theta \in H^{2}$, then $f/\Theta$ is $T_{\phi}$-inner. 
\item Any unit vector belonging to $\ker T_{\overline{\phi}}$ is $T_{\phi}$-inner. 
\end{enumerate}
\end{Proposition}

 If $u$ denotes the inner factor of $\phi$, it is known \cite[p.~108]{MR3526203} that 
 $$\ker T_{\overline{\phi}} = \K_{u}:= (u H^2)^{\perp},$$ the {\em model space} corresponding to $u$. Thus we have the simple corollary. 

\begin{Corollary}\label{IkerT}
If $I$ is any inner function and $u$ is the inner factor of $\phi \in b(H^{\infty})$, then any unit vector from $I \K_u$ is $T_{\phi}$-inner. 
\end{Corollary}

This corollary gives us many specific examples of $T_{\phi}$-inner vectors. For example, if $\lambda \in \D$, the reproducing kernel functions 
$$k_{\lambda}(z) := \frac{1 - \overline{u(\lambda)} u(z)}{1 - \overline{\lambda} z}$$ belong to $\K_u$. In fact, finite linear combinations of these functions are dense in $\K_u$ \cite[Ch.~5]{MR3526203}. Since 
$$\|k_{\lambda}\| = \sqrt{k_{\lambda}(\lambda)} = \sqrt{\frac{1 - |u(\lambda)|^2}{1 - |\lambda|^2}},$$
then 
$$I \sqrt{\frac{1 - |\lambda|^2}{1 - |u(\lambda)|^2}} \frac{1 - \overline{u(\lambda)} u(z)}{1 - \overline{\lambda} z}, \quad \lambda \in \D, \; \; \mbox{$I$ inner},$$ are $T_{\phi}$-inner functions.

When $\phi = u$ is a finite Blaschke product, then the model space $\K_u$ is a certain finite dimensional space of rational functions that are analytic in a neighborhood of $\overline{\D}$ \cite[p.~117]{MR3526203}. Furthermore, as we will see in a moment in Theorem \ref{StessinTinner}, every $T_{u}$-inner function is bounded. However, when $u$ is not a finite Blaschke product then $\K_u$ is infinite dimensional \cite[p.~117]{MR3526203} and, since multiplication by an inner function $I$ is an isometry on $H^2$ (see \eqref{H2IP}), $I \K_u$ is a closed infinite dimensional subspace of $L^2$. By a theorem of Grothendieck, it will contain an unbounded function. Putting this all together, we obtain the following.  

\begin{Corollary}\label{unbounded28475423pqeow}
If the inner factor of $\phi \in b(H^{\infty})$ is not a finite Blaschke product, then there are unbounded $T_{\phi}$-inner functions. 
\end{Corollary}

A specific version of this was pointed out in \cite[p.~103]{MR1437202}.

Of course one needs to discuss the case when $\phi$ is an outer function. Since $\phi H^2$ is dense in $H^2$ \cite[p.~86]{MR3526203}, we see that $\ker T_{\overline{\phi}} = \{0\}$.  In this case, it is not clear that there are {\em any} $T_{\phi}$-inner functions. Indeed, we do not see any obvious ones like $I \ker T_{\overline{\phi}}$ since, in this case,  $\ker T_{\overline{\phi}} = \{0\}$. 

\begin{Example}
Suppose that $\phi$ is the outer function $\phi(z) = 1 + z$ and that $f \in H^2$ is $T_{\phi}$-inner, i.e.,
\[
       \langle T_{\phi}^n f, f \rangle = 0,\ \,\forall n\geq 1.
\]
In other words,
\begin{equation}\label{poipoioiiov}
        \int_{\mathbb{T}} (1 + \xi)^n |f(\xi)|^2\, dm(\xi) = 0,\ \, \forall n\geq 1.
\end{equation}

Then the $L^1$ function $|f|^2$ annihilates $(1+z)^n$ for all $n\geq 1$, along with all their linear combinations.  In particular, $|f|^2$ annihilates
\[
     (1+z)^2 - (1+z) =  1 + 2z + z^2 - 1 - z = z(1+z).
\]
The above observation will be the first step in a proof by induction.  Next,
suppose that $|f|^2$ annihilates $z^k(1+z)$ for all $1\leq k \leq n$.  Then 
\begin{align*}
     z^{n+1}(1+z) &=  (1+z)^{n+2} - \big[  (1+z)^{n+1} - z^{n+1} \big] (1+z).
\end{align*}
By the $T_{\phi}$-inner property of $f$ notice that $|f|^2$ annihilates the first term on the right. It also annihilates  the subtracted expression, by the induction hypothesis (the expression in square brackets is a polynomial of degree $n$).  Thus we  have shown by induction that $|f|^2$ annihilates 
  $ \{z^n(1+z)\}_{n\geq 0}$ (the $n = 0$ case follows from \eqref{poipoioiiov}).
  This means that 
  \begin{equation}\label{55525252a}
  \int_{\T} \xi^{n} (1 + \xi) |f(\xi)|^2 dm(\xi) = 0, \quad n \geq 0,
  \end{equation}
  and by complex conjugation,
  $$ \int_{\T} \overline{\xi}^{n} (1 + \overline{\xi}) |f(\xi)|^2 dm(\xi) = 0, \quad n \geq 0.$$
  A little algebra yields 
  \begin{equation}\label{007yyyyyy}
  \int_{\T} \overline{\xi}^{n + 1} (1 + \xi) |f(\xi)|^2 dm(\xi), \quad n \geq 0.
  \end{equation}
Equations  \eqref{55525252a} and \eqref{007yyyyyy} say that all of the Fourier coefficients of $(1 + \xi) |f(\xi)|^2$ vanish and so $(1 + \xi) |f(\xi)|^2$ is zero. Conclusion: there are no $T_{\phi}$-inner functions when $\phi(z) = 1 +z$.
  %But since $1 + z$ is an outer function, one can show that the weak-$\ast$ closure of the this set is $z H^{\infty}$. In particular this says that 
%$$\int_{\T} \xi^{n} |f(\xi)|^2 dm(\xi) = 0, \quad n \geq 1.$$ Taking complex conjugates shows that 
%$$\int_{\T} \overline{\xi}^{n} |f(\xi)|^2 dm(\xi) = 0, \quad n \geq 1.$$
%By Fourier analysis, $|f|^2$ is a constant $k$. Use \eqref{poipoioiiov} to get 
%$$0 = k \int_{\T} (1 + \xi) dm(\xi) = k$$
%which says that $f$ is the zero function. 

\end{Example}

\section{Inner vectors via the Wold decomposition}

Using some ideas from \cite{MR1437202}, we can use the Wold decomposition \cite{MR0152896} to explore the inner vectors for certain Toeplitz operators. Observe that when $u$ is an inner function the Toeplitz operator $T_u$ is an isometry on $H^2$. Thus the Wold decomposition of $H^2$ with respect to $T_u$ becomes  
$$H^2 = X_{0} \oplus X_1 \oplus T_u X_1 \oplus T_{u}^{2} X_1 \oplus \cdots,$$
where $$X_{0} := \bigcap_{n = 1}^{\infty} T_{u}^{n} H^2 = \{0\}, \quad X_{1} := H^2 \ominus T_{u} H^2 = \K_u.$$
Thus 
$$H^2 = \K_u \oplus u \K_u \oplus u^2 \K_u \oplus \cdots.$$ The above decomposition says that every $f \in H^2$ has a unique expansion as 
\begin{equation}\label{orrfd111919q}
f = F_{0} + u F_{1} + u^2 F_{2} + \cdots, \quad F_{j} \in \K_u.
\end{equation}
 Furthermore, for each integer $N \geq 1$, 
\begin{align*}
\langle u^{N} f, f\rangle & = \Big\langle u^N \sum_{k \geq 0} u^k F_k, \sum_{l \geq 0} u^{l} F_{l}\Big\rangle\\
& = \sum_{k, l \geq 0} \langle u^{N + k - l} F_{k}, F_{l}\rangle\\
& = \sum_{l - k = N} \langle F_{k}, F_{l}\rangle.
\end{align*}
This leads us to the following. 

\begin{Proposition}
A unit vector $f \in H^2$ with expansion 
$$f = F_{0} + u F_{1} + u^2 F_{2} + \cdots, \quad F_{j} \in \K_u,$$
as in \eqref{orrfd111919q}
is $T_{u}$-inner if and only if 
\begin{equation}\label{poeirr34re}
\sum_{k = 0}^{\infty} \langle F_{k}, F_{N + k}\rangle = 0, \quad N \geq 1.
\end{equation}
\end{Proposition}

Though this is just a restatement of the condition for $f$ to be $T_{u}$-inner, it  is useful for producing more tangible examples of $T_{u}$-inner functions. 

\begin{Example}
Choose orthogonal vectors $F_{j}, j \geq 0$ from $\K_u$ so that $\sum_{j \geq 0} \|F_j\|^{2} = 1$. Then the condition \eqref{poeirr34re} is easily satisfied and thus the unit vector $f = \sum_{j \geq 0} u^j F_{j}$ is a $T_{u}$-inner function (as is any inner function times this vector). 
\end{Example}

\begin{Example}\label{432984tre}
If $u(z) = z^n$, then $\K_u = \operatorname{span}\{1, z, z^2, \ldots z^{n - 1}\}$ and the vectors 
$$F_{j} = \frac{z^j}{\sqrt{n}}, \quad 0 \leq j \leq n - 1,$$
satisfy the conditions of the previous example. Thus 
$$f = \sum_{j = 0}^{n - 1} u^j F_{j} = \frac{1}{\sqrt{n}} + \frac{z^{n + 1}}{\sqrt{n}} + \frac{z^{2 n + 2}}{\sqrt{n}} + \frac{z^{3 n + 3}}{\sqrt{n}} + \cdots + \frac{z^{(n - 1)(n + 1)}}{\sqrt{n}}$$
is a $T_{z^n}$-inner vector. 
\end{Example}

\begin{Example}
The previous example can be generalized to a finite Blaschke product 
$$u(z) = \prod_{j = 1}^{n} \frac{z - a_j}{1 - \overline{a_j} z}, \quad a_j \in \D.$$
If we define 
$$F_0(z) = \frac{\sqrt{1 - |a_1|^2}}{1 - \overline{a_1} z},$$
$$F_{1}(z) = \frac{\sqrt{1 - |a_2|^2}}{1 - \overline{a_2} z} \frac{z - a_1}{1 - \overline{a_1} z},$$
$$F_{2}(z) = \frac{\sqrt{1 - |a_3|^2}}{1 - \overline{a_3} z}\frac{z - a_{1}}{1 - \overline{a_1} z} \frac{z - a_2}{1 - \overline{a_2} z},$$
$$\vdots$$
$$F_{n - 1}(z) = \frac{\sqrt{1 - |a_{n}|^2}}{1 - \overline{a_n} z} \prod_{j = 1}^{n - 1} \frac{z - a_{j}}{1 - \overline{a_j} z},$$
one can show that $\{F_0, \ldots, F_{n - 1}\}$ is an orthonormal basis for $\K_u$. Now choose $\alpha_0, \ldots, \alpha_{n - 1} \in \C$ such that $\sum_{j = 0}^{n = 1} |\alpha_j|^2 = 1$. Then
$$f = \sum_{j = 0}^{n - 1} \alpha_j u^{j} F_{j}$$
is $T_{u}$-inner. 
\end{Example}

From Corollary \ref{IkerT} we know, for an inner function $I$, that any unit vector from the set $\{I \ker T_{\overline{u}}: \mbox{$I$ is inner}\}$ is a $T_{\overline{u}}$-inner vector. Perhaps one might think we have equality here. Indeed, sometimes we do. For example, if $u(z) = z$, then $\ker T_{\overline{z}} = \C$ and, as discussed earlier, the $T_{z}$-inner vectors are precisely the inner functions. Here is another positive example of when the unit vectors from $\{I \ker T_{\overline{u}}: \mbox{$I$ is inner}\}$ constitute the complete set of $T_{u}$-inner vectors. 

\begin{Example}
If the inner function $u$ is the single Blaschke factor
$$u(z) = \frac{z - a}{1 - \overline{a} z}, \quad a \in \D,$$ one can show \cite[Ch.~5]{MR3526203} that 
$$\ker T_{\overline{u}} = \K_u = \C \frac{1}{1 - \overline{a} z}.$$
As shown in \cite{InnerRKS}, the $T_{u}$-inner vectors are
$$I \frac{\sqrt{1 - |a|^2}}{1 - \overline{a} z}, \quad \mbox{$I$ inner}.$$
\end{Example}

However, in general, the unit vectors from $\{I \ker T_{\overline{u}}: \mbox{$I$ is inner}\}$ form a proper subset of the $T_{u}$-inner vectors. One can see this with the following example. 

\begin{Example}
Using the technique from Example \ref{432984tre}, we see that when $u(z) = z^n$ the vector 
$$f = \frac{1}{\sqrt{2}} + \frac{z^{n + 1}}{\sqrt{2}}$$ is $T_{u}$-inner. However, $f$ is not of the form $I g$, where $I$ is inner and $g \in \K_u$. This follows from the fact that $f$ is outer and does not belong to $\K_u = \operatorname{span}\{1, z, z^2, \ldots, z^{n - 1}\}$.
\end{Example}

The papers \cite{MR1437202, MR1969797} yield a description of the $T_u$-inner vectors. From the Wold decomposition \eqref{orrfd111919q} we see that any $f \in H^2$ can be written as 
$$f = \sum_{k = 0}^{\infty} F_k u^k.$$
If $\{v_j\}_{j \geq 1}$ is an orthonormal basis for $\K_u$, then we can expand things a bit further and write 
\begin{align*}
f & = \sum_{k = 0}^{\infty} F_{k} u^k\\
& = \sum_{k = 0}^{\infty} u^k \Big(\sum_{j \geq 1} c_{j, k} v_j\Big)\\
& = \sum_{j \geq 1} v_j \Big(\sum_{k = 0}^{\infty} c_{j, k} u^k\Big).
\end{align*}
Observe that 
$$\sum_{j \geq 1} |c_{j, k}|^2 = \|F_{k}\|^2$$ and that 
\begin{align*}
\|f\|^2 & = \sum_{k = 0}^{\infty} \|F_{k}\|^2\\
& = \sum_{k = 0}^{\infty} \sum_{j \geq 1} |c_{j, k}|^2\\
& = \sum_{j \geq 1} \sum_{k = 0}^{\infty} |c_{j, k}|^2.
\end{align*}
Thus for each $j$, $\sum_{k \geq 0} |c_{j, k}|^2 < \infty$ and so 
$$f_{j}(z)  = \sum_{k = 0}^{\infty} c_{j, k} z^k$$ defines a function in $H^2$ (square summable power series). By the Littlewood subordination principle \cite[p.~126]{MR3526203},  $f_{j} \circ u$ also belongs to $H^2$. 

Thus every unit vector $f \in H^2$ has the unique representation 
\begin{equation}\label{009rsdsfqqqx}
f(z)= \sum_{j \geq 1} v_{j}(z) f_{j}(u(z)),
\end{equation}
 where $f_j \in H^2$ with $\sum_{j \geq 1} \|f_{j}\|^2 < \infty$, and $\{v_j\}_{j \geq 1}$ is an orthonormal basis for $\K_u$. Furthermore, as observed in \cite[Prop.~1]{MR1437202} (and can be proved using the above calculation), if 
 \begin{equation}\label{bbb6sfaa}
 f = \sum_{j \geq 1} v_j f_j(u), \quad g = \sum_{j \geq 1} v_j g_j(u),
 \end{equation}
  as in \eqref{009rsdsfqqqx}, then 
  \begin{equation}\label{9yyyyyxy}
 \langle f, g\rangle = \sum_{j \geq 1} \langle f_j, g_j\rangle.
 \end{equation}

\begin{Theorem}\label{StessinTinner}
A unit vector $f$ written as in \eqref{009rsdsfqqqx} is $T_{u}$-inner if and only if 
$$\sum_{j \geq 1}^{\infty} |f_j(\xi)|^2 = 1$$ for almost every $\xi \in \T$. 
\end{Theorem}

\begin{proof}
Here is the original proof from \cite{MR1437202}. 
With 
$$f = \sum_{j \geq 1} v_j f_j(u),$$ and $n \geq 1$, \eqref{9yyyyyxy} yields
\begin{align}
\langle T_{u}^{n} f, f \rangle \nonumber
& = \langle f u^n, f\rangle\\ \nonumber
& = \big\langle \sum_{j} v_j u^n f_j(u), \sum_{k} v_k f_k(u)\big\rangle\\ \nonumber
& = \sum_{j \geq 1} \langle z^n f_j, f_j\rangle\\ \nonumber
& = \sum_{j \geq 1} \int_{\T} \xi^n |f_j(\xi)|^2 dm(\xi)\\ 
& = \int_{\T} \xi^n\Big( \sum_{j \geq 1}|f_j(\xi)|^2 \Big) dm(\xi).\label{11199274}
\end{align}
Then $\langle T_{u}^{n} f, f\rangle = 0$ for all $n = 1, 2, \ldots$ if and only if, by Fourier analysis, 
$\sum_{j \geq 1} |f_j|^2$ is constant almost everywhere. But since we assuming that $f$ is a unit vector, we see, by putting $n = 0$ in \eqref{11199274}, that 
%$$1 = \|f\|^2 = \sum_{j \geq 1} \|f_j\|^2 = \sum_{j \geq 1} \int_{\T} |f_j|^2 dm = \int_{\T} \sum_{j \geq 1} |f_j|^2 dm$$ 
 $\sum_{j \geq 1} |f_j|^2 = 1$ almost everywhere. 
\end{proof}

When $u$ is a finite Blaschke product, then $\K_u$ is finite dimensional.  In this case \eqref{009rsdsfqqqx} is finite and each basis vector $v_j$ is a rational function that is analytic in a neighborhood of $\overline{\D}$ \cite[Ch.~5]{MR3526203}. From here it follows that every $T_u$-inner vector is a bounded function. Contrast this with Corollary \ref{unbounded28475423pqeow} which says that when $u$ is not a finite Blaschke product there are always $T_u$-inner vectors that are unbounded functions. 

The two papers \cite{MR1437202, MR1969797} go further and discuss an ``inner-outer'' factorization of any $f \in H^2$ in terms of $T_{u}$-inner and $T_{u}$-outer vectors. They also discuss the concept of $T_u$-inner in $H^p$, for $p > 1$, along with some properties of the norms of $T_{u}$-inner vectors as well as their growth near $\T$.

\section{Inner vectors via the operator-valued Poisson kernel}

We can rephrase the language of inner vectors for Toeplitz operators in terms  of operator-valued Poisson kernels \cite{MR2018850}. Moreover, using this new language, we can extend our discussion to inner vectors for contractions on Hilbert spaces. For $\lambda \in \D$ and $\xi \in \T$, define 
\begin{equation}\label{PK}
P_{\lambda}(\xi) := \frac{1}{1 - \overline{\lambda} \xi} + \frac{1}{1 - \lambda \overline{\xi}}- 1
\end{equation}
and observe that this can be written as 
$$P_{\lambda}(\xi) = \frac{1 - |\lambda|^2}{|\xi - \lambda|^2},$$
which is the standard Poisson kernel. Classical theory says that for any $g \in L^1 = L^1(\T, m)$ the function 
$$\int_{\T} P_{\lambda}(\xi) f(\xi) dm(\xi)$$ is harmonic on $\D$ with 
\begin{equation}\label{wuerge}
\lim_{r \to 1^{-}} \int_{\T} P_{r \zeta}(\xi) f(\xi) dm(\xi) = f(\zeta)
\end{equation}
 for almost every $\zeta \in \T$. Furthermore, if $\mu$ is a finite complex measure on $\T$, we have
 \begin{equation}\label{98304ipweofdjsv}
 \int_{\T} P_{\lambda}(\xi) d\mu(\xi) = \widehat{\mu}(0) + \sum_{n \geq 1} \widehat{\mu}(n) \lambda^{n} + \sum_{n \geq 1} \widehat{\mu}(-n) \overline{\lambda}^{n},
 \end{equation}
 where 
 $$\widehat{\mu}(n) := \int_{\T} \overline{\xi}^{n} d \mu(\xi), \quad n \in \mathbb{Z},$$
 are the Fourier coefficients of $\mu$. 
 We will now discuss an operator version of the Poisson kernel.

For a contraction $T$ on a Hilbert space $\mathcal{H}$, we imitate the formula in \eqref{PK} and define, for $\lambda \in \D$, the {\em operator-valued Poisson kernel }$K_{\lambda}(T)$ as 
$$K_{\lambda}(T) := (I - \lambda T^{*})^{-1} + (I - \overline{\lambda} T)^{-1} - I.$$ By the spectral radius formula, notice how $\sigma(T) \subset \overline{\D}$ and thus the formula for $K_{\lambda}(T)$ above makes sense. A computation with Neumann  series will show that for $r \in [0, 1)$ and $\theta \in [0, 2 \pi)$
\begin{equation}\label{oosdfp9}
K_{r e^{i \theta}}(T) = \sum_{n = 0}^{\infty} r^{n} e^{i n \theta} T^{* n} + \sum_{n = 0}^{\infty} r^n e^{-i n \theta} T^{n} - I.
\end{equation}

The operator identity 
$$K_{\lambda}(T) = (I - \overline{\lambda} T)^{-1} (I - |\lambda|^2 T T^{*}) (I - \lambda T^{*})^{-1}$$
from  \cite[Lemma 2.4]{MR2018850} shows that for each $\mathbf{x} \in \mathcal{H}$
$$\langle K_{\lambda}(T) \mathbf{x}, \mathbf{x}\rangle \geq 0, \quad \lambda \in \D.$$ Moreover, the function 
$$\lambda \mapsto \langle K_{\lambda}(T) \mathbf{x}, \mathbf{x}\rangle$$ is harmonic on $\D$. Hence, a classical harmonic analysis result of Herglotz (\cite[p. 10]{Duren} or \cite[p.~17]{MR3526203}) produces a unique positive finite Borel measure $\mu_{T, \mathbf{x}}$ on $\T$ such that 
\begin{equation}\label{ee4rfsds}
\langle K_{\lambda}(T) \mathbf{x}, \mathbf{x}\rangle = \int_{\T} P_{\lambda}(\zeta) d \mu_{T, \mathbf{x}}(\zeta).
\end{equation}
Since $K_{0}(T) = I$ we have
$$1 = \langle \mathbf{x}, \mathbf{x}\rangle = \langle K_{0}(T) \mathbf{x}, \mathbf{x}\rangle  = \int_{\T} d \mu_{T, \mathbf{x}}$$ and so $\mu_{T, \mathbf{x}}$ is a probability measure. 

As we defined for Toeplitz operators earlier in Definition \ref{oow0099898}, we say that a unit vector $\mathbf{x}$ is {\em $T$-inner} if 
$$\langle T^{n} \mathbf{x},  \mathbf{x}\rangle = 0, \quad n \geq 1.$$ Note that $\mathbf{x}$ is $T$-inner if and only if $\mathbf{x}$ is $T^{*}$-inner. From \eqref{oosdfp9} we see that  $\mathbf{x}$ is $T$-inner if and only if $\langle K_{\lambda}(T) \mathbf{x}, \mathbf{x}\rangle = 1$ for all $\lambda \in \D$, or equivalently, 
$$1 = \int_{\T} P_{\lambda}(\zeta) d \mu_{T, \mathbf{x}}(\zeta), \quad \lambda \in \D.$$ By \eqref{98304ipweofdjsv} this is equivalent to the condition $\mu_{T, \mathbf{x}} = m$. This gives us the following.

\begin{Proposition}\label{p354refvcxdw}
Suppose that $T$ is a contraction on a Hilbert space $\mathcal{H}$ and $\mathbf{x}$ is unit vector in $\mathcal{H}$. Then $\mathbf{x}$ is $T$-inner if and only if $\mu_{T, \mathbf{x}} = m$, where $\mu_{T, \mathbf{x}}$ is defined as in \eqref{ee4rfsds}.
\end{Proposition}

For an inner function $u$, note that $T_u$ is an isometry, hence a contraction. Thus we can apply the above analysis to $\mu_{T_{u}, f}$.

\begin{Proposition}\label{bsfghy}
If $$f = \sum_{j \geq 1} v_j f_j(u)$$ is a vector from $H^2$ as in \eqref{009rsdsfqqqx}, then 
\begin{equation}\label{ppsdpfpsdf}
d \mu_{T_u, f} = \sum_{j \geq 1} |f_j|^2 \; dm.
\end{equation}
\end{Proposition}

\begin{proof} If
$$f = \sum_{j \geq 1} v_j f_j(u),$$ then 
$$\|f\|^2 = \sum_{j \geq 1} \|f_j\|^2 = \sum_{j \geq 1} \int_{\T} |f_j|^2 dm = \int_{\T} \sum_{j \geq 1} |f_j|^2 dm $$ and  the calculation used to prove Theorem \ref{StessinTinner} yields
$$\langle T_{u}^{n} f, f\rangle = \int_{\T} \xi^n\Big( \sum_{j \geq 1}|f_j(\xi)|^2 \Big) dm(\xi),$$
$$\langle T_{u}^{* n} f, f\rangle = \int_{\T} \overline{\xi}^n\Big( \sum_{j \geq 1}|f_j(\xi)|^2 \Big) dm(\xi).$$
From here we observe 
\begin{align*}
\int_{\T} P_{\lambda}(\xi) d \mu_{T_{u}, f}(\xi) & = 
 \langle K_{\lambda}(T_u) f, f\rangle\\
& = \sum_{n \geq 0} \lambda^n \langle T_{u}^{*n} f, f\rangle + \sum_{n \geq 0} \overline{\lambda}^n \langle T_{u}^{n} f, f\rangle - \langle f, f\rangle.\\
& = \sum_{n \geq 0} \lambda^n \int_{\T} \overline{\xi}^n \Big(\sum_{j \geq 1} |f_j(\xi)|^2\Big)dm(\xi)\\
& \quad  \quad + \sum_{n \geq 0} \overline{\lambda}^n  \int_{\T} \xi^n \Big(\sum_{j \geq 1} |f_j(\xi)|^2\Big)dm(\xi)\\
& \quad  \quad \quad - \sum_{j \geq 1} \int_{\T} |f_j(\xi)|^2 dm\\
& = \int_{\T} (\frac{1}{1 - \lambda \overline{\xi}} + \frac{1}{1 - \overline{\lambda} \xi} - 1) \sum_{j \geq 1} |f_j(\xi)|^2 dm(\xi)\\
& = \int_{\T} P_{\lambda}(\xi) \sum_{j \geq 1} |f_j(\xi)|^2 dm(\xi)
\end{align*}
Now use the uniqueness of the Fourier coefficients of a measure along with  \eqref{98304ipweofdjsv} to  obtain \eqref{ppsdpfpsdf}. 
\end{proof}

Notice how this gives us another way of thinking about Theorem \ref{StessinTinner}: a unit vector $f \in H^2$ is $T_{u}$-inner if and only if $\mu_{T_{u}, f}  = m$.  

This brings us to an interesting related question. One can also show that for any $f, g \in H^2$, we can define the harmonic function $\langle K_{\lambda}(T_u) f, g\rangle$ on $\D$ and prove this function also has bounded integral means. This yields, via Herglotz's theorem,  a complex valued measure $\mu_{T_u, f, g}$ on $\T$ for which 
\begin{equation}\label{vvx04trdas11qlfjgldfv}
\langle K_{\lambda}(T_u) f, g\rangle = \int_{\T} P_{\lambda}(\xi) d \mu_{T, f, g}(\xi), \quad \lambda \in \D.
\end{equation}
See \cite[Prop.~2.6]{MR2018850} for details. 
A similar calculation used to prove Proposition \ref{p354refvcxdw} shows that 
\begin{equation}\label{lllhhehe}
d \mu_{T_u, f, g} = \sum_{j \geq 1} f_j \overline{g_j} \; dm.
\end{equation}
 In the above formula, $f_j$ and $g_j$ come from the representations of $f$ and $g$ from \eqref{bbb6sfaa}. A general result from \cite{MR2050137} says that given any $F \in L^1$ and a non-constant inner function $u$ that is not an automorphism, there are $f, g \in H^2$ for which 
\begin{equation}\label{94378reoigfjlekadwq}
F(\zeta) = \frac{d \mu_{T_u, f, g}}{dm}(\zeta)
\end{equation}
$m$-almost everywhere. In the language of  \cite{MR2050137} this says that any $F \in L^1$ can be ``factored through $T_{u}$''. Equivalently stated, using \eqref{lllhhehe} and \eqref{94378reoigfjlekadwq}, we have 
$$F(\zeta) = \sum_{j \geq 1} f_j(\zeta) \overline{g_j(\zeta)}.$$
This is an interesting representation for $L^1$ functions and  a refinement of the  one from \cite{MR2050137}. 

\begin{Question}
Proposition \ref{bsfghy} shows that when $\phi$ is an inner function and $f, g \in H^2$, then $d \mu_{T_{\phi}, f, g}$ is absolutely continuous with respect to $m$. When $\phi \in b(H^{\infty})$ is this still the case? For this to be true we would need to know that 
$\langle \phi^n f, g\rangle, n \geq 1$, are the Fourier coefficients of an $L^1$ function. 
\end{Question}

\begin{comment}
\begin{Question}
Given an inner $u$, what non-negative $F \in L^1$ can be represented as 
$$F(\zeta)  = \sum_{j \geq 1} |f_j(\zeta)|^2$$ for an $f \in H^2$ written as 
$$f = \sum_{j \geq 1} v_j f_j(u).$$
\end{Question}

When $u(z) = z$, then the expansion in \eqref{009rsdsfqqqx} is trivial and the question becomes when can a positive $L^1$ function be written as $F = |f|^2$ for some $f \in H^2$. Certainly $\log F$ must be integrable and using outer functions this is the only condition for such a representation of $F  = |f|^2$ for an $f \in H^2$. When $u$ is a finite Blaschke product of degree $d$, the question become one of when we can write a non-negative $F \in L^1$ as 
$$F = \sum_{j = 1}^{d} |f_j|^2$$ on $\T$ for some $f_1, \ldots, f_d \in H^2$. 
\end{comment}

\section{Inner vectors via Clark measures} 

 For any fixed $\alpha \in \T$ and inner function $u$, the function 
$$z \mapsto \frac{1  - |u(z)|^2}{|\alpha - u(z)|^2} = \Re\Big(\frac{\alpha + u(z)}{\alpha - u(z)}\big)$$ is a positive harmonic function on $\D$. Thus by Herglotz's theorem, there is a unique positive measure $\sigma_{\alpha}$ on $\T$ for which 
$$\frac{1  - |u(z)|^2}{|\alpha - u(z)|^2} = \int_{\T} P_{\lambda}(\xi) d \sigma_{\alpha}(\xi).$$ The family of measures 
$\{\sigma_{\alpha}: \alpha \in \T\}$ is called the family of {\em Clark measures} corresponding to $u$. Let us record some important facts about this family of measures. Proofs can be found in \cite{MR2215991}.

First, one can use the fact that $u$ is an inner function, along with standard harmonic analysis, to prove that each $\sigma_{\alpha}$ is singular with respect to $m$. Second, if $E_{\alpha}$ is defined to be the set of $\xi \in \T$ for which 
$$\lim_{r \to 1^{-}} u(r \xi) = \alpha,$$ then $E_{\alpha}$ is a Borel subset of $\T$ with
\begin{equation}\label{carrier}
\sigma_{\alpha}(\T \setminus E_{\alpha}) = 0.
\end{equation}
 In other words, $\sigma_{\alpha}$ is ``carried'' by $E_{\alpha}$. From this we also see that the measures $\{\sigma_{\alpha}: \alpha \in \T\}$ are singular with respect to each other. Third, a  beautiful disintegration theorem of Aleksandrov says that if $g \in L^1$ then for $m$-almost every $\alpha  \in \T$, integral 
$$\int_{\T} g(\xi) d\sigma_{\alpha}(\xi)$$ is well defined. Moreover this almost everywhere defined function 
$$\alpha \mapsto \int_{\T} g(\xi) d\sigma_{\alpha}(\xi)$$ is integrable with respect to $m$ and 
\begin{equation}\label{ADT}
\int_{\T} \Big(\int_{\T} g(\xi) d\sigma_{\alpha}(\xi)\Big)dm(\alpha) = \int_{\T} g(\zeta) dm(\zeta).
\end{equation}

Using Clark measures, we can use a technique from \cite{MR1969797} to compute a formula for $\langle K_{\lambda}(T_{u}) f, f\rangle$ along with the measure $d \mu_{T_u, f}/d m$. This gives us another way to think about the formula \eqref{94378reoigfjlekadwq}. The result here is the following. 

\begin{Theorem}\label{rueiydas}
For an inner function $u$ and $f \in H^2$ we have 
$$d \mu_{T_u, f}(\alpha)  = \Big(\int_{\T} |f(\xi)|^2 d\sigma_{\alpha}(\xi) \Big)\; dm(\alpha).$$
\end{Theorem}

\begin{proof}
For any $f \in H^2$ use the formulas from \eqref{carrier} and \eqref{ADT} to obtain 
\begin{align*}
\langle T_{u}^{n} f, f\rangle & = \int_{\T} |f(\xi)|^2 u(\xi)^{n} dm(\xi)\\
& = \int_{\T} \Big(\int_{\T} |f(\xi)|^2 u(\xi)^{n} d \sigma_{\alpha}(\xi)\Big) dm(\alpha)\\
& = \int_{\T} \Big(\int_{\T} |f(\xi)|^2 \alpha^n d \sigma_{\alpha}(\xi)\Big) dm(\alpha)\\
& = \int_{\T} \alpha^n \Big(\int_{\T} |f(\xi)|^2 d \sigma_{\alpha}(\xi)\Big) dm(\alpha).
\end{align*}
In a similar way 
$$\langle T_{u}^{*n} f, f\rangle = \int_{\T} \overline{\alpha}^n \Big(\int_{\T} |f(\xi)|^2 d \sigma_{\alpha}(\xi)\Big) dm(\alpha).$$
Now follow the proof of Proposition \ref{bsfghy} to get 
\begin{align*}
& \int_{\T} P_{\lambda}(\xi) d \mu_{T_{u}, f}(\xi)\\
& = 
 \langle K_{\lambda}(T_u) f, f\rangle\\
& = \sum_{n \geq 0} \lambda^n \langle T_{u}^{*n} f, f\rangle + \sum_{n \geq 0} \overline{\lambda}^n \langle T_{u}^{n} f, f\rangle - \langle f, f\rangle\\
& = \int_{\T}\left( \Big(\frac{1}{1 - \lambda \overline{\alpha}} + \frac{1}{1 - \overline{\lambda} \alpha} - 1\Big) \Big(\int_{\T} |f(\xi)|^2 d\sigma_{\alpha}(\xi)\Big)\right) dm(\alpha)\\
& = \int_{\T} P_{\lambda}(\alpha) \Big(\int_{\T} |f(\xi)|^2 d\sigma_{\alpha}(\xi)\Big) dm(\alpha).
\end{align*}
Use \eqref{98304ipweofdjsv} along with the uniqueness of Fourier coefficients of a measure to compute the proof. 
\end{proof}

Combing  Theorem \ref{rueiydas} and Proposition \ref{p354refvcxdw} yields the following result from \cite{MR1969797}.

\begin{Corollary}
A unit vector $f \in H^2$ is $T_{u}$-inner if and only if 
$$\int_{\T} |f(\xi)|^2 d\sigma_{\alpha}(\xi) = 1$$ for $m$-almost every $\alpha \in \T$. 
\end{Corollary}

Recall the notation from \eqref{vvx04trdas11qlfjgldfv} that for a given inner function $u$ and $f, g \in H^2$
$$\langle K_{\lambda}(T_u) f, g\rangle = \int_{\T} P_{\lambda}(\xi) d \mu_{T_u, f, g}(\xi).$$
Moreover, if $\operatorname{deg}(u) \geq 2$, any $F \in L^1$ can be written as $d \mu_{T_u, f, g}(\xi)/dm$ for some $f, g \in H^2$. Here is another way of thinking about this via Clark measures. The same argument used to prove Theorem \ref{rueiydas} shows that 
\begin{equation}\label{r438rfuibvbdjf}
d \mu_{T_u, f, g} = \int_{\T} f(\xi) \overline{g(\xi)} d\sigma_{\alpha}(\xi) \; dm
\end{equation}
 Since any $F \in L^1$ is equal to $d \mu_{T_u, f, g}/dm$ for some $f, g \in H^2$ \cite{MR2050137}, we see that any $F \in L^1$ can be written as 
$$F(\alpha) =  \int_{\T} f(\xi) \overline{g(\xi)} d\sigma_{\alpha}(\xi).$$ This Clark measure viewpoint has the additional feature, via Aleksandrov's theorem, that 
\begin{align*}
\int_{\T} F(\alpha) dm(\alpha) & = \int_{\T} \Big(\int_{\T} f(\xi) \overline{g(\xi)} d\sigma_{\alpha}(\xi)\Big) dm(\alpha)\\
& = \int_{\T} f(\zeta) g(\zeta) dm(\zeta).
\end{align*}

\begin{Example}\label{ooooo9o9}
If $u$ is a finite Blaschke product of degree $d$ and $\alpha \in \T$, then one can compute (see \cite[p.~209]{MR2215991} for the details) the Clark measure to be 
$$d\sigma_{\alpha} = \sum_{j = 1}^{d} \frac{1}{|u'(\zeta_j)|} \delta_{\zeta_j} ,$$
where $\zeta_1, \ldots, \zeta_d$ are the $d$ distinct solutions to the equation $u(z) = \alpha$ and $\delta_{\zeta_j}$ is the unit point pass as $\zeta_j$. The denominators in the above expression may look troublesome but at the end of the day we have  $u' \not = 0$ on $\T$.  By Theorem \ref{rueiydas} we see that 
$$\frac{d \mu_{T_u, f}}{d m}(\alpha) = \int_{\T} |f(\xi)|^2 d\sigma_{\alpha}(\xi) = \sum_{j = 1}^{d} \frac{|f(\zeta_j)|^2}{|u'(\zeta_j)|}.$$ Thus the criterion for a unit vector $f \in H^2$ to be a $T_u$-inner vector is that the above sum is equal to $1$ for $m$-almost every $\alpha \in \T$. 

Furthermore, by \eqref{r438rfuibvbdjf}, given $F \in L^1$, there are $f, g \in H^2$ so that 
$$F(\alpha) = \sum_{j = 1}^{d} \frac{f(\zeta_j) \overline{g(\zeta_j)}}{|u'(\zeta_j)|}$$ for $m$-almost every $\alpha \in \T$.  This formula appears in \cite{MR2050137}.
\end{Example}

\begin{Example}
Let us apply this to the simple case where $u(z) = z^2$. Given any $\alpha \in \T$, the two solutions $\zeta_1, \zeta_2$ to the equation $z^2 = \alpha$ are 
$$\zeta_1 = e^{i \arg \alpha/2}, \quad \zeta_2 = - e^{i \arg \alpha/2}.$$ Thus the condition that a unit $f$ is a $T_{z^2}$-inner vector becomes 
$$|f(e^{i \arg \alpha/2})|^2 + |f(-e^{i \arg \alpha/2})|^2 = 2, \quad \mbox{$m$-a.e. $\alpha \in \T$}.$$ Furthermore, given any $F \in L^1$, there are $f, g \in H^2$ for which 
$$F(\alpha) = \tfrac{1}{2} f(e^{i \arg \alpha/2}) \overline{g(e^{i \arg \alpha/2})} + \tfrac{1}{2} f(-e^{i \arg \alpha/2}) \overline{g(-e^{i \arg \alpha/2})}.$$ This second fact was first observed in \cite{MR2050137}.
\end{Example}

\begin{Example}
Consider the atomic inner function 
$$u(z) = \exp\Big(\frac{z + 1}{z - 1}\Big).$$
For a fixed $t \in [0, 2 \pi)$, the solutions to $u(z) = e^{i t}$ are 
$$\zeta_k = \frac{ i (t + 2 \pi k) + 1}{i (t + 2 \pi k) - 1}, \quad k \in \mathbb{Z}.$$
Noting that 
$$|u'(\zeta_k)| = \frac{2}{|\zeta_k - 1|^2},$$
a similar computation as in Example \ref{ooooo9o9} shows that 
$$d\sigma_{e^{i t}} = \frac{1}{2} \sum_{k \in \mathbb{Z}} \delta_{\zeta_k} |\zeta_k - 1|^2.$$
Thus 
\begin{align*}
\frac{d \mu_{T_{u}, f}}{d m}(e^{i t}) & = \int_{\T} |f(\xi)|^2 d\sigma_{e^{i t}}(\xi)\\
& = \frac{1}{2} \sum_{k \in \mathbb{Z}} |f(\zeta_k)|^2 |\zeta_k - 1|^2\\
& = \sum_{k\in\mathbb{Z}}\Big| f\Big( \frac{i[t+2\pi k]+1}{i[t+2\pi k]-1}  \Big)  \Big|^2\, \frac{2}{|i(t+2\pi k) -1|^2}.
\end{align*}
To create a $T_{u}$-inner function, we need to find a unit vector $f \in H^2$ so that the above expression is equal to one for almost every $t$. Let us create a specific example of when this happens. In fact we can even make $f$ unbounded. We already knew we could do this from Corollary \ref{unbounded28475423pqeow} but our example  below will be explicit, while the proof of Corollary \ref{unbounded28475423pqeow} needed Grothendieck's theorem and is not an explicit construction. 

To see how to do this, fix $\beta \in (\tfrac{1}{2},1)$, and let $a_k$, $k\in \mathbb{Z}$, be the collection of coefficients 
\begin{equation}\label{akassump}
         a_k =  \frac{1}{ 1+|k|^{\beta}}.
\end{equation}
Note that $\sum_{k \in \mathbb{Z}} |a_k|^2 < \infty$.

Let $I_k$ be the indicator function of the interval $[-\pi + 2\pi k, \pi +2\pi k)$, $k\in\mathbb{Z}$.  Now define $F$ on $\mathbb{T}$ by
\[
       F(e^{i\theta}) := \sqrt{2} \sum_{k\in\mathbb{Z}} \frac{a_k}{e^{i\theta}-1}\,I_k\Big(  i\frac{1+e^{i\theta}}{1-e^{i\theta}} \Big).
\]
Then
\begin{align*}
     \int_{\mathbb{T}}|F|^2\, dm &=  2 \int_{-\pi}^{\pi} |F(e^{i\theta})|^2 \,\frac{d\theta}{2\pi} \\
     &= 2\sum_{k\in\mathbb{Z}}\int_{-\pi}^{\pi}  \frac{|a_k|^2}{|e^{i\theta}-1|^2}\,I_k\Big(  i\frac{1+e^{i\theta}}{1-e^{i\theta}} \Big)\,\frac{d\theta}{2\pi}\\
     &= 2\sum_{k\in\mathbb{Z}}\int_{-\infty}^{\infty}  |a_k|^2 \frac{| it-1 |^2}{2^2}\,I_k(t)\,\frac{2\,dt}{2\pi|it-1|^2}\\
     &= \sum_{k\in\mathbb{Z}}\int_{-\pi}^{\pi}  |a_k|^2 \big| i[t+2\pi k]-1 \big|^2\,\frac{dt}{2\pi|i[t+2\pi k]-1|^2}\\
     &= \sum_{k\in\mathbb{Z}} |a_k|^2 < \infty,
     %&= \sum_{k\in\mathbb{Z}} \frac{1}{(1 + |k|^{\beta})^2} < \infty,
   %  &\leq |a_0|^2 + 2 \sum_{k=1}^{\infty} \frac{1}{36\pi^2 k^2}\\
\end{align*}
i.e., $F$ is square integrable on $\mathbb{T}$ with 
\begin{equation}\label{3489ureiogfdsa}
\|F\|^{2} =  \sum_{k \in \Z} |a_k|^2.
\end{equation}

Next we establish that $\log|F|$ is integrable.  We'll need the following estimates, which hold for all $k\neq 0$.  First note that for $k \neq 0$,
\begin{align*}
      |a_k| \big| i(t+2\pi k) - 1  \big| &= |a_k| \big(  [\pi + 2\pi |k|]^2 + 1 \big)^{1/2}  \\
           &\geq |a_k| \cdot 2\pi |k| \\
           &\geq \frac{2\pi |k|}{1+|k|^{\beta}}\\
           &\geq 1.
\end{align*}
Consequently, for $k \neq 0$ and $t \in [-\pi, \pi)$,
\begin{align*}
    \Big|\log\Big( |a_k| \big| i(t+2\pi k) - 1\Big)\Big| &= \log { |a_k| | i(t+2\pi k)- 1|}  \\
       &\leq \log\frac{ | i(\pi+2\pi |k|)- 1|}{1 + |k|^{\beta} } \\
       &\leq \log\frac{( [ 2\pi(|k|+1/2)]^2 + 1)^{1/2}}{1 + |k|^{\beta} } \\
%              &\leq \log\frac{( [ 2\pi(|k|+1/2)]^2 + 1)^{1/2}}{1 + |k|^{\beta} } \\
%       &\leq \log\frac{1}{\pi |a_k| |k|} \\
       &\leq \log\frac{( [ 2\pi(|k|+|k|/2)]^2 + |k|^2)^{1/2}}{|k|^{\beta} } \\
       &\leq \log ( |k|^{1-\beta} \sqrt{9 \pi^2 + 1}).
\end{align*}

We now have
\begin{align*}
     & \int_{\mathbb{T}}\big| \log |F| \big| \, dm\\ &=  \int_{-\pi}^{\pi} \Big| \log|F(e^{i\theta})|\Big| \,\frac{d\theta}{2\pi} \\
      &= \sum_{k\in\mathbb{Z}}\int_{-\pi}^{\pi}  \Big|\log\frac{|a_k|\sqrt{2}}{|e^{i\theta}-1|}\Big|\,I_k\Big(  i\frac{1+e^{i\theta}}{1-e^{i\theta}} \Big)\,\frac{d\theta}{2\pi}\\
      &= \sum_{k\in\mathbb{Z}}\int_{-\infty}^{\infty} \Big| \log\big(|a_k| \big| it-1 \big|\sqrt{2}/2\big)\Big| \,I_k(t)\,\frac{dt}{2\pi|it-1|^2}\\
      &= \sum_{k\in\mathbb{Z}}\int_{-\pi}^{\pi} \Big| \log\big(6\pi |k|^{1-\beta} \big| i[t+2\pi k]-1 \big|/\sqrt{2}\big)\Big|\,\frac{dt}{2\pi|i[t+2\pi k]-1|^2}.\\
%      &\leq 2\pi \log\big(|a_0|\big|i\pi-1\big|\big) + 2 \sum_{k=1}^{\infty} \frac{ k^{\beta}}{2 \pi^2 k^2}, \\
\end{align*}
The series is summable, because the terms behave like $(\log |k|)/|k|^{2}$.

It follows that there exists an outer function $g \in H^2$ with radial limit function satisfying $|g| = |F|$ almost everywhere on $\mathbb{T}$, namely
\[
     g(z) := \exp\Big(\int_{\mathbb{T}} \frac{e^{i\theta}+z}{e^{i\theta}-z} \log |F(e^{i\theta})|\, dm(e^{i\theta})\Big).
\]  

Finally, let $J$ be any classical inner function, and define $f = gJ$.  Then
\begin{align*}
 \frac{d \mu_{T_{u}, f}}{d m}(e^{i t})
       &= \sum_{k\in\mathbb{Z}}\Big| f\Big( \frac{i[t+2\pi k]+1}{i[t+2\pi k]-1}  \Big)  \Big|^2\, \frac{2}{|i(t+2\pi k) -1|^2}\\ &= 
              \sum_{k\in\mathbb{Z}}\Big| F\Big( \frac{i[t+2\pi k]+1}{i[t+2\pi k]-1}  \Big)  \Big|^2\, \frac{2}{|i(t+2\pi k) -1|^2}\\
               &= 
              \sum_{k\in\mathbb{Z}}\frac{|a_k( i[t+2\pi k]-1)\sqrt{2}|^2}{2^2}\, \frac{2}{|i(t+2\pi k) -1|^2}\\
   &=  \sum_{k\in\mathbb{Z}} |a_k|^2.
\end{align*}
Notice from \eqref{3489ureiogfdsa} that 
$$ \frac{d \mu_{T_{u}, f}}{d m}(e^{i t}) = \|F\|^2$$ and so one can scale $F$ so that it (and hence $f$) is a unit vector
This also gives us $d \mu_{T_u, f}/dm(e^{i t}) = 1$ for almost every $t$. Any such $f$ will be a $T_u$-inner function. 

As a bonus, we get that the $f$ we just constructed is unbounded. To see this, note that $F$ is unbounded, since for $\theta$ approaching zero, $F(e^{i\theta})$ takes values
\[
      F\Big( \frac{i[t+2\pi k]+1}{i[t+2\pi k]-1}  \Big) = \frac{a_k}{1 - \frac{i[t+2\pi k]+1}{i[t+2\pi k]-1}} = \frac{-i[t+2\pi k]+1}{2 + 2|k|^{\beta}}
\]
where $t \in [-\pi,\pi)$.  Since $\beta < 1$, this expression is unbounded as $|k| \to \infty$.
\end{Example}

\section{Inner vectors in model spaces}

In this section we depart slightly from Toeplitz operators on $H^2$ to the related topic of compressions of Toeplitz operators on model spaces. For an inner function $\Theta$, recall the model space $\K_{\Theta} = (\Theta H^2)^{\perp}$. An important operator to study here is the {\em compressed shift operator} 
$$S_{\Theta}: \K_u \to \K_u, \quad S_{\Theta} f = P_{\Theta}(z f),$$ where $P_{\Theta}$ is the orthogonal projection of $L^2$ onto $\K_u$. This operator is used to model a certain class of contraction operators on Hilbert space \cite[Ch.~9]{MR3526203} -- hence the use of the phrase ``model space.''

As a generalization of our discussion of classifying the $T_z$-inner vectors in $H^2$, one can ask for a description of the $S_{\Theta}$-inner vectors in $\K_{\Theta}$, i.e., those unit vectors $f \in \K_{\Theta}$ for which 
$$\langle S_{\Theta}^{n} f, f\rangle = 0, \quad n \geq 1.$$
Before continuing, let us make a few comments about $S_{\Theta}$. For the proofs, see \cite[Ch.~9]{MR3526203}. First note that since $S_{\Theta}$ is a compression of $T_z$ to $\K_{\Theta}$ we have the identity 
$$S_{\Theta}^{n} = P_{\Theta} T_{z^n}|_{\K_{\Theta}}.$$
Furthermore, we have the adjoint formula 
$$S_{\Theta}^{*} = T_{\overline{z}}|_{\K_{\Theta}}.$$
For any $\phi \in H^{\infty}$ there is the functional calculus for $S_{\Theta}$ which allows us to define 
$$\phi(S_{\Theta}) = P_{\Theta} T_{\phi}|_{\K_{\Theta}}$$ along with the adjoint formula 
$$\phi(S_{\Theta})^{*} = P_{\Theta} T_{\overline{\phi}}|_{\K_{\Theta}}.$$

One can actually compute the $S_{\Theta}$-inner vectors with the following result from \cite[p.~177]{MR3526203}. 

\begin{Theorem}\label{poeriepotr}
Any $S_{\Theta}$-inner function is an inner function. Moreover, $\K_{\Theta}$ contains an inner function if and only if $u(0) = 0$ and the inner functions belonging to $\K_{\Theta}$ are precisely the inner divisors of $\Theta(z)/z$. 
\end{Theorem}

So now the question becomes the following. 

\begin{Question}
What are the $\phi(S_{\Theta})$-inner functions?
\end{Question}

As we did before with Toeplitz operators, we focus our attention on the case where $\phi$ is inner. It is clear that the inner vectors for $\phi(S_{\Theta})$ are the same as those for $\phi(S_{\Theta})^{*}$. As observed with an analogous result in Proposition \ref{6iehtg43tnnncknkn}, we see that any (unit) vector in $\ker \phi(S_{\Theta})^{*}$ is a $\phi(S_{\Theta})^{*}$-inner vector. It is well-known \cite{MR3526203} that (assuming $\phi$ is an inner function)
$$\ker \phi(S_{\Theta})^{*} = \K_{\Theta} \cap \K_{\phi} = \K_{\operatorname{gcd}(\Theta, \phi)},$$
where $\operatorname{gcd}(\Theta, \phi)$ is the greatest common inner divisor of the inner functions $\Theta$ and $\phi$. 

At this point, it might the case that $\operatorname{gcd}(\Theta, \phi)$ is a unimodular constant function whence $\K_{\operatorname{gcd}(\Theta, \phi)} = \{0\}$ and it is not clear as to whether or not there are any $\phi(S_{\Theta})$-inner vectors.

\begin{Question}
We know that if $\operatorname{gcd}(\Theta, \phi)$ is non-constant, then there are $\phi(S_{\Theta})$-inner vectors. Is the converse true? 
\end{Question} 

For the special case where $\phi|\Theta$, let us find a class of $\phi(S_{\Theta})$-inner vectors.
Define 
$$I := \frac{\Theta}{\phi}$$ and observe from a result in \cite{MR3720929} that an analytic  function $g$  on $\D$ multiplies $\K_{\phi}$ to $\K_{\Theta}$ if and only $g \in \K_{z I}$. Recall from Theorem \ref{poeriepotr} that the inner functions in $\K_{z I}$ are precisely the inner divisors of $I$.  Here is our result about some of  the $\phi(S_{\Theta})$-inner vectors. 

\begin{Theorem}
With the notation above, any unit vector from 
$$\{v \K_{\phi}: v|I\}$$ is a $\phi(S_{\Theta})$-inner vector. 
\end{Theorem}

\begin{proof}
Let $f$ be a unit vector from $\K_{\phi}$ and note that $v f \in \K_{\Theta}$  and hence $P_{\Theta}(v f) = v f$.  Thus for all $n \geq 1$ we have 
\begin{align*}
\langle (\phi(S_{\Theta}))^{n} (v f), v f\rangle & = \langle P_{\Theta}(\phi^n f v), v f\rangle\\
& = \langle \phi^{n} v f, P_{\Theta}(v f)\rangle\\
& = \langle \phi^n v f, v f\rangle\\
& = \langle \phi^n f, f\rangle\\
& = \langle f, T_{\overline{\phi}}^{n} f\rangle.
\end{align*}
But since $f \in \K_{\phi} = \ker T_{\overline{\phi}}$, this last quantity is equal to zero. This shows that $v f$ is a $\phi(S_{\Theta})$-inner vector. 
\end{proof}

When $\Theta(0) = 0$ and $\phi(z) = z$, notice how this recovers Theorem \ref{poeriepotr}. At the other extreme, notice that when $\phi  = \Theta$ then $I$ is a unimodular constant inner function and the theorem above yields $\K_{\Theta}$ as the complete set of $T_{\Theta}$-inner functions. Of course this result is obvious once one realizes that $\langle T_{\Theta} f, f\rangle =0$ for any $f \in \K_{\Theta}$ by the definition of the model space $\K_{\Theta} = (\Theta H^2)^{\perp}$.

Also observe that one can relax the assumption that $\phi|\Theta$ and set $I = u/\operatorname{gcd}(\Theta, \phi)$ and give a more general version of the theorem above.

\bibliography{references}

\end{document}